\documentclass[11pt]{amsart}
\usepackage{amssymb}
\usepackage{amsmath}
\usepackage{amsthm}
\usepackage{graphicx}
\usepackage[textwidth=17cm, textheight=24.7cm, marginratio=1:1]{geometry}

\usepackage[english]{babel}
\usepackage[utf8]{inputenc}
\usepackage[T1]{fontenc}
\usepackage{todonotes}
\usepackage{xcolor}
\usetikzlibrary{arrows, cd}

\newcommand{\CC}{\mathbb{C}}

\newcommand{\OO}{\mathcal{O}}
\renewcommand{\AA}{\mathbb{A}}

\renewcommand{\phi}{\varphi}
\renewcommand{\epsilon}{\varepsilon}

\newcommand{\GL}{\mathrm{GL}}

\newcommand{\Hilb}{\mathrm{Hilb}}
\newcommand{\Hom}{\mathrm{Hom}}

\newcommand{\pr}{\operatorname{pr}}%

\usepackage[backref=page]{hyperref}
\hypersetup{
    colorlinks,
    linkcolor={red!80!black},
    citecolor={blue!80!black},
    urlcolor={blue!80!black}
}
\usepackage{setspace}
\usetikzlibrary{matrix,arrows, cd, calc}  
\newtheorem{theorem}{Theorem}

\newtheorem{proposition}[theorem]{Proposition}
\theoremstyle{definition}
\newtheorem{definition}[theorem]{Definition}

\numberwithin{equation}{section}

\begin{document}

\title{Behrend's function is not constant on $\Hilb^n(\mathbb{A}^3)$}
\author{J.~Jelisiejew, M.~Kool, and R.~F.~Schmiermann}
\maketitle
\begin{abstract}
    We prove the statement in the title for $n\geq 24$.
\end{abstract}

\section{Introduction}

The Behrend function $\nu_X$ of a scheme $X$ of finite type over $\CC$ is a constructible function introduced in the seminal paper \cite{Behrend__function}. It plays a central role in Donaldson--Thomas theory: if $X$ is proper and admits a symmetric perfect obstruction theory, then the degree of the virtual cycle $[X]^{\mathrm{vir}}$ equals the $\nu_X$-weighted Euler characteristic of $X$.
It has been conjectured that $\nu_{\Hilb^n(\mathbb{A}^3)}$ is constant, equal to
$(-1)^n$, see \cite[Conj.~A]{Ricolfi_notes}, \cite[Problem~XXVI]{Jelisiejew__open_problems}, although the conjecture is older than these references. A reason for this expectation is the fact that the value of the Behrend function is indeed $(-1)^n$ on all smooth points of the smoothable component and on all monomial ideals of $\Hilb^n(\mathbb{A}^3)$ \cite{MNOP}.
Constancy would imply that $\Hilb^n(\mathbb{A}^3)$ is generically
reduced, as explained in~\cite{Ricolfi_notes}, which would solve a major open problem \cite{Fog68, Ame10}.
The Behrend function is notoriously difficult to compute. For $X$
zero-dimensional, an algebraic method was developed by
Graffeo--Ricolfi~\cite{Graffeo_Ricolfi}. In general, a powerful localization
method was introduced by Behrend-Fantechi:

\begin{theorem}[{\cite[Thm.~3.4]{Behrend__Fantechi}}]\label{ref:isolatedPointBehrend:thm}
    Let $X$ be an affine $\mathbb{C}^*$-scheme with an isolated fixed point $P$. Assume that $X$ admits an equivariant symmetric obstruction theory.
    Then $\nu_X(P) = (-1)^{\dim_{\CC} T_X|_P}$.
\end{theorem}
The assumption that $X$ is affine can be removed by using (a very special case of)~\cite[Thm.~4.4]{AHR}. In our applications $X$ embeds into a smooth variety, so one can use~\cite{Sumihiro} instead.

For more on symmetric equivariant obstruction theories, we refer
to~\cite{Behrend__Fantechi, Behrend__function, Ricolfi__Modern_Enumerative_Geometry}.
The Hilbert scheme $\Hilb^n(\mathbb{A}^3)$ admits such an obstruction theory provided that
$\mathbb{C}^*$ lies inside the Calabi-Yau torus $\left\{ (t_1,t_2,t_3)\in (\mathbb{C}^*)^3\ |\
t_1t_2t_3 = 1 \right\}$.
For any $\mathbb{C}^*$ (in the usual 3-dimensional torus) acting on $\Hilb^n(\mathbb{A}^3)$ such that the fixed points are isolated, the fixed points correspond to monomial ideals.
For these, the value of the Behrend function is known to be $(-1)^n$, by combining Theorem \ref{ref:isolatedPointBehrend:thm} with the virtual localization calculation in \cite{MNOP}.
For a smooth point $P$ of the smoothable component of $\Hilb^n(\mathbb{A}^3)$, we have $\nu_{\Hilb^n(\mathbb{A}^3)}(P) = (-1)^{n}$.
The tools above are enough to prove that
$\nu_{\Hilb^n(\mathbb{A}^3)}$ is constant for $n\leq 4$\footnote{As explained by Andrea Ricolfi in a private conversation.}, where $n=4$ is the first interesting case.
Other than that, little is known.
Here we provide a first example where
$\nu_{\Hilb^n(\mathbb{A}^3)}$ is \emph{not} constant.

\begin{theorem}[{Theorem~\ref{ref:main:thm}}]\label{ref:intro:mainthm}
    For $n\geq 24$, there exist points of $\Hilb^n(\mathbb{A}^3)$ on which
    the Behrend function is equal to $-(-1)^n$. In particular, the Behrend
    function is not constant.
    For $n=24$ one such point is $[I]$, where $I\subset \CC[x,y,z]$ is the following ideal
    \[
        I = \left( (x^2) + (y, z)^2 \right)^2 + (y^3 - x^3z).
    \]
\end{theorem}

Although the point $P$ of the theorem is a fixed point of a suitable $\CC^*$-action, it is \emph{non-isolated}, so we cannot apply Theorem \ref{ref:isolatedPointBehrend:thm} directly. We circumvent this by taking a quotient by another transverse free $\CC^*$-action such that $P$ becomes an isolated fixed point on the quotient. Since this transverse $\CC^*$ is not in the Calabi-Yau torus, the usual symmetric obstruction theory is not equivariant and it is not clear whether it descends to the quotient. We construct another superpotential on the quotient, using an ``average'' of the usual superpotential, and apply Theorem~\ref{ref:isolatedPointBehrend:thm} to the symmetric obstruction theory obtained from this other superpotential.

Our argument is quite general and does not use much about the Hilbert scheme,
only some equivariance properties of the superpotential,
see~\S\ref{sec:Hilbert}.

\subsection{Parity of the tangent space}

The constancy of the Behrend function is related to another conjecture
disproved earlier this year. The conjecture predicted that for every point
$P\in \Hilb^n(\mathbb{A}^3)$ we have
\begin{equation}\label{eq:parity}
    \dim_{\mathbb{C}} T_{\Hilb^n(\mathbb{A}^3)}|_P \equiv n \pmod{2}.
\end{equation}
The relation between the two conjectures is Theorem \ref{ref:isolatedPointBehrend:thm}. The equation~\eqref{eq:parity} was established for monomial ideals
in~\cite{MNOP} and recently for homogeneous ideals (even with respect to some
non-standard gradings) in~\cite{Ramkumar_Sammartano_parity}.
However, a counterexample to~\eqref{eq:parity} was given
by Giovenzana--Giovenzana--Graffeo--Lella~\cite{Graffeo_Giovenzana_Lella}.

Although their counterexample does not directly fit the hypotheses in
Section~\S\ref{sec:Hilbert}, our example in Theorem~\ref{ref:intro:mainthm} is
very closely related and the present paper would not exist without their example.

\subsection{Open questions}

It is unknown whether $\Hilb^{n}(\mathbb{A}^3)$ is irreducible for $n = 12$\footnote{There
is an ongoing project by Sema Gunturkun to prove this.}, see
also~\cite[Problem~V]{Jelisiejew__open_problems}. Proving irreducibility
for $n=24$ seems hopeless, however in~\cite{Graffeo_Giovenzana_Lella_two} the authors show that the point described in Theorem~\ref{ref:intro:mainthm} lies in the smoothable component. Due to the connection with
reducedness~\cite{Ricolfi_notes} it is natural to ask: \emph{Is
    $\Hilb^{24}(\mathbb{A}^3)$ reduced at $[I]$?} We do not know the answer. In general, it is not known whether $\Hilb^n(\mathbb{A}^3)$ is reduced, see~\cite[Problem~I]{Jelisiejew__open_problems}. See also~\cite{Graffeo_Giovenzana_Lella_two} for results and questions for the Quot scheme.

\subsection*{Acknowledgements.} We thank Gavril Farkas and Rahul Pandharipande for organizing the ``Workshop Hilbert schemes of points'' on 28-9-2023 at the Humboldt University, which prompted this work. In particular, Pandharipande asked whether the Giovenzana--Giovenzana--Graffeo--Lella counterexample also provides a counterexample to constancy of the Behrend function. We warmly thank Alessio Sammartano for the breakfast discussion between him and the first- and second-named authors, which laid the foundations for the main idea of this paper. We also warmly thank Michele Graffeo, in particular for the lunch discussion he had with the first-named author in Warsaw. We thank Jim Bryan, Andrea T.~Ricolfi, Balazs Szendr\H{o}i, and Richard P.~Thomas for comments on earlier versions of this article or related discussions.

JJ is supported by National Science Centre grant 2020/39/D/ST1/00132. This work is partially supported by  the Thematic Research Programme ``Tensors: geometry, complexity and quantum entanglement'', University of Warsaw, Excellence Initiative – Research University and the Simons Foundation Award No.~663281 granted to the Institute of Mathematics of the Polish Academy of Sciences for the years 2021--2023. MK and RFS are supported by NWO Grant VI.Vidi.192.012 and MK is also supported by ERC Consolidator Grant FourSurf 101087365.

\section{Behrend function in the equivariant setup}\label{sec:Behrend}

All schemes considered in this paper are of finite type over $\mathbb{C}$.
Let $P$ be a fixed point for an action of an algebraic group $G$ on a scheme $X$. We say that $P$ is \emph{isolated} if $P$ is equal \emph{as a scheme} to a connected component of the fixed locus $X^G$. If $P$ is smooth and $G$ is linearly reductive, then $P$ is isolated if and only if $\{P\}$ is a connected component of the topological space $X^G$. However, even on a normal variety $X$, it may happen that a singular point $P$ is isolated topologically, but $X^G$ is nonreduced at $P$.

For a scheme $X$ over $\mathbb{C}$, the Behrend function~\cite{Behrend__function} is a constructible function $\nu_X\colon X(\mathbb{C})\to \mathbb{Z}$ which gives rise to an ``Euler characteristic description'' of Donaldson-Thomas invariants in the case $X$ is a moduli space of stable sheaves on a smooth projective Calabi-Yau 3-fold. The value $\nu_X(x)$ depends only on (any) \'etale neighbourhood of $x\in X$  \cite[p.1309]{Behrend__function} and for a morphism $f\colon X\to Y$ smooth at $x\in X$  we have $\nu_X(x) = (-1)^{\dim f^{-1}(f(x))}\nu_Y(f(x))$, see~\cite[p.1309]{Behrend__function} and~\cite[\href{https://stacks.math.columbia.edu/tag/039P}{Tag 039P}]{stacks-project}.
It is generally very hard to determine the value of the Behrend function at a given point,
see~\cite{Graffeo_Ricolfi} for results in the case when $X$ is zero-dimensional.

\subsection{Critical loci}

For generalities on critical loci and symmetric obstruction theories, we refer
to~\cite{Behrend__Fantechi, Ricolfi__Modern_Enumerative_Geometry}.
We recall some basics for completeness.
For a smooth variety $X$ and a global 1-form $\omega\in H^0(\Omega_X)$, we have a vanishing
scheme $Z(\omega)$. Over an open subset $U$ on which we have a trivialization $\Omega_{X}|_U \simeq
\bigoplus_{i=1}^{\dim X} \OO_U dz_i$ with $\omega = \sum f_idz_i$, the vanishing
scheme $Z(\omega)\cap U = Z(\omega|_U)$ is given by $f_1 =  \cdots = f_{\dim X} = 0$.
For a \emph{smooth} morphism $\varphi\colon Y\to X$ and the form $df$, for some global regular function $f\colon X\to \mathbb{C}$, we have
\begin{equation}\label{eq:pullback}
    \varphi^{-1}(Z(df)) = Z(d(f\circ\varphi)).
\end{equation}

\newcommand{\fbar}{\overline{f}}%
The following notation will be useful to us.
\begin{definition}\label{ref:semiinvariant:def}
Let $G$ be an algebraic
group, $X$ a $G$-variety, and $f\colon X\to \mathbb{C}$ a morphism. Let $\chi\colon
G\to\mathbb{C}^*$ be a character. We say that $f$ is
a \emph{$G$-semi-invariant with weight $\chi$} if $f(g\cdot P) = \chi(g)f(P)$
for all $g\in G$ and $P \in X$.
\end{definition}
\begin{proposition}[Critical locus of a semi-invariant]\label{ref:criticalLocusOfSemiinvariant:prop}
    Let $G$ be an algebraic group and $\chi\colon G\to \mathbb{C}^*$ a
    character, nontrivial on the connected component of $1_G\in G$.\footnote{In particular, the connected component of $1_G \in G$ is not zero-dimensional.} Let $B$ be a smooth variety and $f\colon G\times B\to \mathbb{C}$ a
    $G$-semi-invariant with weight $\chi$, where $G$ acts trivially on $B$. Take $\fbar := f(1_G, -)\colon B\to \mathbb{C}$.
    Then
    \[
        Z(df) = \pr_B^{-1}(Z(d\fbar)\cap Z(\fbar)),
    \]
    where $\pr_B\colon G\times B\to B$ is the projection.
\end{proposition}
We observe that if $\chi$ above is trivial, then $Z(df) =
\pr_B^{-1}(Z(d\fbar))$, so the claim fails.
\begin{proof}
    By equivariance of $f$, we have $f(g, b) = \chi(g)f(1, b) =
    \chi(g)\fbar(b)$. Therefore, we also have
    \begin{equation}\label{eq:dfsplitting}
        df = \chi d\fbar + \fbar d\chi.
    \end{equation}
    The character $\chi\colon G\to \mathbb{C}^*$ is a surjection of algebraic
    groups, hence it is a smooth morphism onto a curve, so
    its differential is nowhere vanishing and $Z(\fbar d\chi) =
    Z(\fbar)$. 
    The function $\chi$ is also nowhere
    vanishing, so $Z(\chi d\fbar) = Z(d\fbar)$.
    The two summands
    in~\eqref{eq:dfsplitting} are elements of distinct summands of the direct sum
    $\pr_G^*\Omega_{G} \oplus \pr_B^*\Omega_B \simeq \Omega_{G\times B}$, thus
    \[
        Z(df) = Z(\chi d\fbar) \cap Z(\fbar d\chi) = \pr_B^{-1}\left(Z(d\fbar) \cap Z(\fbar)\right),
    \]
    as claimed.
\end{proof}
For examples of functions $f$ such that $Z(df)$ does not contain $Z(f)$ and for related pathologies, see for example~\cite{Saito__Quasihomogeneous}
and~\cite[App.~A]{Maulik_Pandharipande_Thomas}.

\subsection{General setup}

Let $f \colon A \to \mathbb{C}$ be a regular function on a smooth variety $A$.
Let $M:=Z(df)$ be the critical locus of $f$ and $P \in M$ a closed point. We
are interested in the value $\nu_M(P)$ of the Behrend function.

For any $d \geq 0$, let $G$ be a $d$-dimensional affine algebraic group acting regularly on $A$. Let $T_0 \cong \mathbb{C}^*$ be an algebraic torus acting regularly on $A$, and commuting with $G$, such that:
\begin{itemize}
    \item $f$ is $T_0$-invariant and $f$ is a $G$-semi-invariant (Definition~\ref{ref:semiinvariant:def}),
    \item $\mathrm{Stab}_G(P) = \{1\}$ and $P$ lies in the $T_0$-fixed locus,
    \item the $T_0$-fixed part of the tangent space at $P$, $T_M|_P^{T_0} \subset T_M|_P$, is $d$-dimensional.
\end{itemize}
Suppose further that there is another linearly reductive group $H$
acting regularly on $A$, and commuting with $G$ and $T_0$, such that $P$ lies in the $H$-fixed locus and $f$ is an $H$-semi-invariant with character $\chi \colon H \to \CC^*$ which is \emph{nontrivial} on the connected component of $1_H \in H$.

\begin{theorem}\label{ref:BehrendFuncEquivariant:thm}
    In the setup above, we have $\nu_M(P) = (-1)^{\dim_{\CC} T_M|_P}$.
\end{theorem}
\begin{proof}
First we reduce to the case of $A$ being a product $G \times B$.
By assumption, the stabilizer $\mathrm{Stab}_{G\times H\times T_0}(P)$ is equal to $H\times T_0$, hence is linearly reductive. A strong version of Luna's \'etale slice theorem \cite[Thm.~4.5]{AHR} provides an affine scheme $B$ with a $(H\times T_0)$-action, a $(H \times T_0)$-fixed point $Q \in B$, and a $(G \times H \times T_0)$-equivariant \'etale morphism $\mu \colon G \times B \to A$ mapping $P' := (1_G, Q)$ to $P$. Since $A$ is smooth, it follows that $B$ is a smooth variety.
Let $f' := f \circ \mu \colon G \times B \to \CC$. Since
$\mu^{-1}(Z(df)) = Z(df')$, we obtain an \'etale morphism $\mu
\colon M':=Z(df') \to M$ and $\nu_{M'}(P') = \nu_M(P)$. We note that $f'$ is
$T_0$-invariant, and is a $G$- and $H$-semi-invariant with the same characters as
before. Moreover $\mathrm{Stab}_{G}(P') = \{1\}$, $P'$ lies in the $T_0$- and
$H$-fixed loci, and $T_{M'}|_{P'}^{T_0}$ is $d$-dimensional.

Let $\psi = f'(1_G, -)\colon B\to \mathbb{C}$. We claim that
\begin{equation}\label{eq:accordanceOfcritical}
    \pr_B^{-1}(Z(d\psi)) = Z(df').
\end{equation}
Then $N:=Z(d\psi) \subset B$ has a $T_0$-equivariant symmetric perfect obstruction theory by \cite[Ex.~1.19]{Behrend__Fantechi}. Moreover, since $T_{M'}|_{P'}^{T_0}$ is $d$-dimensional, it follows that $T_N|_Q^{T_0}$ is 0-dimensional and $Q \in N$ is an isolated $T_0$-fixed point, so we may apply Theorem~\ref{ref:isolatedPointBehrend:thm}. Therefore, assuming~\eqref{eq:accordanceOfcritical}, we get that $Z(df') \simeq
G \times N$ and so we obtain
\[
\nu_M(P) = \nu_{M'}(P') = (-1)^d \nu_N(Q) = (-1)^{d + \dim_{\CC} T_N|_{Q}} = (-1)^{\dim_{\CC} T_{M'}|_{P'}} = (-1)^{\dim_{\CC} T_{M}|_{P}},
\]
which concludes the proof.

It remains to prove~\eqref{eq:accordanceOfcritical}.
Let $\sigma_{H}\colon H \times B\to B$ be the action map, so
$\sigma_{H}(h, b) = h \cdot b$. The action map is smooth, since it is a composition of the isomorphism $(h, b) \mapsto (h, h\cdot b)$ and a
projection $\pr_B\colon H \times B \to B$ where $H$ is smooth.
Let $\sigma\colon G \times H \times B \to G \times B$ be the induced smooth
map, defined by $\sigma(g, h, b) = (g, h\cdot b)$.
We have the following situation
\[
    \begin{tikzcd}
        G \times H \times B \ar[d, "\sigma"]\\
        G \times B \ar[r, "\textrm{\'etale}"]\ar[rr, "f'", bend left]\ar[d, "\pr_B"] & A \ar[r] & \mathbb{C}\\
        B \ar[rr, "\psi"] & & \mathbb{C}
    \end{tikzcd}
\]
Consider the functions $f'' = f'\circ \sigma$ and $\psi'' = \psi\circ \pr_B \circ
\sigma$. 
They are both $(G\times H)$-semi-invariants with characters that are nontrivial on the connected component of $(1_G,1_H) \in G \times H$. Moreover, the functions $f''(1_G, 1_{H}, -) = \psi''(1_G, 1_{H}, -)$ are both equal
to $\psi$. Applying Proposition~\ref{ref:criticalLocusOfSemiinvariant:prop} to
each of them, we obtain 
\[
    Z(df'') = p^{-1}\Big(Z(d\psi)\cap Z(\psi)\Big) = Z(d\psi''),
\]
where $p\colon G\times H \times B\to B$ is the projection to $B$.
The maps $\sigma$ and $\pr_B$ are smooth, so
\[
    \sigma^{-1}(Z(df')) = Z(df'') = Z(d\psi'') = \sigma^{-1}(Z(d(\psi\circ \pr_B)) = \sigma^{-1}\pr_B^{-1} Z(d\psi).
\]
We obtain that the preimages of both sides
of~\eqref{eq:accordanceOfcritical} under a smooth surjective map $\sigma$ are equal, hence both sides of~\eqref{eq:accordanceOfcritical} are equal. The map $\sigma$ has a natural section, so the equality of both sides also follows by restricting to this section.
\end{proof}

\section{Applications to Hilbert schemes}\label{sec:Hilbert}

Let $M:=\Hilb^n(\mathbb{A}^3)$ be the Hilbert scheme of $n$ points on
$\mathbb{A}^3$. We recall the well-known fact~\cite{Szendroi, Szendroi_Dimca,
Ricolfi__Modern_Enumerative_Geometry} that there exists a smooth variety $A$
with regular function $f \colon A \to \mathbb{C}$ such that $M = Z(df)$. Let $V$ be an $n$-dimensional complex vector space and consider the quiver in Figure \ref{fig1}.
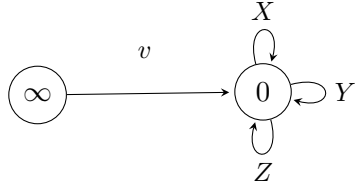
\begin{figure}[h] \label{fig1}
\begin{tikzpicture}[>=stealth,->,shorten >=2pt,looseness=.5,auto]
  \matrix [matrix of math nodes,
           column sep={3cm,between origins},
           row sep={3cm,between origins},
           nodes={circle, draw, minimum size=7.5mm}]
{ 
|(A)| \infty & |(B)| 0 \\         
};
\tikzstyle{every node}=[font=\small\itshape]
\path[->] (B) edge [loop above] node {$X$} ()
              edge [loop right] node {$Y$} ()
              edge [loop below] node {$Z$} ();

\node [anchor=west,right] at (-0.3,0.55) {$v$};              
\draw (A) to [bend left=0,looseness=1] (B) node [midway] {};
\end{tikzpicture}
\caption{Quiver for the non-commutative Hilbert scheme.}
\end{figure}

We consider representations of this quiver with dimension vector $(1,n)$, i.e.~we put the vector space $\mathbb{C}$ at the node $\infty$ and $V$ at the node $0$. 
Representations with this dimension vector correspond to elements of $W = \mathbb{C}^3 \otimes \mathrm{End}(V) \oplus V$.
The group $\GL(V)$ acts on $W$ by
$$
g \cdot (X,Y,Z,v) := (g X g^{-1},gYg^{-1},gZg^{-1},g v).
$$
Let $\widetilde{A} \subset W$ be the open subset consisting of representations satisfying
$$
\CC\langle X,Y,Z \rangle \cdot \langle v \rangle_{\CC} = V.
$$ 
Then $\GL(V)$ acts freely on $\widetilde{A}$ and the non-commutative Hilbert scheme is defined as the smooth variety
$$
A:=\mathrm{ncHilb}^{n}(\AA^3) = \widetilde{A} / \GL(V).
$$
The regular function $\widetilde{f} \colon \widetilde{A} \to \CC$, $\widetilde{f}(X,Y,Z,v) = \mathrm{tr}(X[Y,Z])$ is invariant under the action of $\GL(V)$ and descends to a regular function $f \colon A \to \CC$, and $M = Z(df)$.
Furthermore, we have a regular action of $(\CC^*)^3$ on $\widetilde{A}$ defined by
\[
(t_1,t_2,t_3) \cdot (X,Y,Z,v) = (t_1X,t_2Y,t_3Z,v),
\]
which descends to an action on $A$. Clearly, we have
\[
f((t_1,t_2,t_3)\cdot P) = t_1t_2t_3 f(P), \quad \forall (t_1,t_2,t_3) \in (\CC^*)^3, \, P \in A.
\]

\begin{theorem}\label{ref:main:thm}
For $n\geq 24$, there exist points of $\Hilb^n(\mathbb{A}^3)$ on which
    the Behrend function is equal to $-(-1)^n$. In particular, the Behrend
    function is not constant.
    For $n=24$ one such point is $[I]$, where $I$ is the following ideal in $\CC[x,y,z]$
    \[
        I = \left( (x^2) + (y, z)^2 \right)^2 + (y^3 - x^3z).
    \]
\end{theorem}
\begin{proof}
    We denote $M = \Hilb^{24}(\mathbb{A}^3)$ and $P = [I]$. By hand or using \emph{Macaulay2} one
    shows that $\CC[x,y,z] / I$ is $24$-dimensional and $T_M|_P =
    \Hom_{\CC[x,y,z]}(I,\CC[x,y,z] / I)$ is $99$-dimensional, see
    Appendix~\ref{sec:appendix}. In particular, $(-1)^{\dim_{\CC} T_M|_P} \neq
    (-1)^{24}$.

    We want to apply Theorem~\ref{ref:BehrendFuncEquivariant:thm} with $P =
    [I]$, $M = \Hilb^{24}(\mathbb{A}^3)$ and $A,f$ as above. For this, we define the tori
\begin{align*}
T_0 := \{ (t^2,t,t^{-3}) \in (\CC^*)^3 \}, \quad G := \{ (1,1,t) \in (\CC^*)^3 \}, \quad H := \{ (t,t,1) \in (\CC^*)^3 \}.
\end{align*}
Clearly $f$ is $T_0$-invariant and is a $G$-semi-invariant with weight $1$. Moreover, $f$ is an
$H$-semi-invariant with weight $2$. We also note that $\mathrm{Stab}_G(P) =
\{1\}$ and $P$ is $T_0$-fixed and $H$-fixed. Again by hand or
\emph{Macaulay2}, we check that $T_M|_P^{T_0}$ is 1-dimensional,
Appendix~\ref{sec:appendix}. The monomial $x^3z$ maps to a socle element in the quotient $\CC[x,y,z]/I$ and $y^3-x^3z$ is a minimal generator of $I$, so the transformation that sends $y^3 - x^3z$ to $x^3z$ and other minimal generators to zero extends to a nonzero tangent vector in $T_M|_P$. This vector
spans the one-dimensional space $T_M|_P^{T_0}$.

Therefore, by Theorem~\ref{ref:BehrendFuncEquivariant:thm}, we have $\nu_M(P)
= -1$. Clearly, any $Q \in M$ corresponding to a reduced scheme
is a smooth point with $(3 \cdot 24)$-dimensional tangent space,
so it satisfies $\nu_M(Q) = 1$, thus the Behrend function is not constant on $M$.
To obtain the claim for $n\geq 24$, add a tuple of disjoint points in $\AA^3$ to $[I]$: \'etale locally the neighbourhood of the scheme obtained in this way is a product of the neighbourhood of a disjoint tuple of points and the neighbourhood of $[I]$. By the properties of the Behrend function on \'etale neighbourhoods and smooth maps, we conclude the result.
\end{proof}

\appendix
\goodbreak
\section{Macaulay2 code}\label{sec:appendix}
\begin{verbatim}
kk = QQ;
S = kk[x, y, z, Degrees=>{{1, 2}, {2, 1}, {3, -3}}];
I = ideal mingens((ideal(x^2) + ideal(y^2, y*z, z^2))^2 + ideal(y^3 - x^3*z));
tgI = Hom(I, S^1/I);
assert(rank source basis(S^1/I) == 24); -- Hilbert scheme of 24 points
assert(rank source basis(tgI) == 99); -- 99-dimensional tangent space
for i in -1000..1000 do (
    ress := rank source basis({i, 0}, tgI);
    if ress != 0 then
        print((i,0), ress);
);
\end{verbatim}
Macaulay2 technical note: it is necessary to work with an auxiliary grading, since $(2, 1, -3)$
raises a \texttt{no heft vector} error. Alternatively, one can work with the
purely non-negative degrees $\left\{ \left\{ 1, 0 \right\}, \left\{ 1, 1
\right\}, \left\{ 0, 3 \right\} \right\}$.

\bibliographystyle{alpha}
\bibliography{refs}
\end{document}